\documentclass[12pt, reqno]{amsart}
\usepackage{amsmath, amsthm, amscd, amsfonts, amssymb, graphicx, color}
\usepackage[bookmarksnumbered, colorlinks, plainpages]{hyperref}
\hypersetup{colorlinks=true,linkcolor=red, anchorcolor=green, citecolor=cyan, urlcolor=red, filecolor=magenta, pdftoolbar=true}
\usepackage[applemac]{inputenc}

\newtheorem{theorem}{Theorem}[section]
\newtheorem{lemma}[theorem]{Lemma}
\newtheorem{proposition}[theorem]{Proposition}
\newtheorem{corollary}[theorem]{Corollary}
\theoremstyle{definition}
\newtheorem{definition}[theorem]{Definition}

\newtheorem{property}[theorem]{Property}
\newtheorem{remark}[theorem]{Remark}
\numberwithin{equation}{section}

\newtheorem{ex}{Example}[section]

\newcommand{\modint}{\displaystyle\copy\tratto\kern-10.4pt\int\limits}

\newcommand{\D}{\mathbb D}

\def\R{{\rm I\!R}}
\def\N{{\rm I\!N}}

\def\OV{\overline}
\def\GG{G\Gamma}
\def\G{\Gamma}
\def\O{\Omega}
\def\log{{\rm Log\,}}
\def\a{\alpha}
\def\b{\beta}
\def\d{\delta}

\def\f{\varphi}
\def\g{\gamma}
\def\p{\partial}
\def\s{\sigma}
\def\t{\theta}

\def\LEQ{\leqslant}
\def\GEQ{\geqslant}

\def\Log{{\rm Log\,}}

\def\WT{\widetilde}
\def\DST{\displaystyle}
\def\pr{{\bf Proof:\\ }}
\def\HF{\hfill{$\diamondsuit$\\ }}

\def\NN{\nonumber}

\def\up{{\frac1p}}

\def\GG{G\Gamma}
\def\G{\Gamma}
\def\O{\Omega}

\def\a{\alpha}
\def\b{\beta}
\def\d{\delta}
\def\f{\varphi}
\def\g{\gamma}
\def\p{\partial}
\def\s{\sigma}
\def\t{\theta}

\def\ln{{\rm Log\,}}
\def\Log{{\rm Log\,}}

\def\DST{\displaystyle}

\def\dist{{\rm dist\,}}

\def\up{\frac1p}

\def\R{{\rm I\!R}}

\date{\today}                                       
\def\OVO{{\overline \Omega}}
\def\vws{{\rm v.w.s.\,}}

  \def\LUOdpO{{L^1(\O,\dist(x,\p\O))}}
  \def\LUO{{L^1(\Omega)}}
  \def\dx{\,dx}

\begin{document}

\title[Some  new results related to $\GG$-spaces and interpolation]{Some  new results related to Lorentz $\GG$-spaces\\ and interpolation}

\author[ I. Ahmed, A. Fiorenza, M.R. Formica, A. Gogatishvili \MakeLowercase{and} J.M. Rakotoson]{{\bf   Irshaad Ahmed}$^1$, {\bf Alberto Fiorenza}$^2$, {\bf Maria Rosaria Formica}$^3$,  {\bf Amiran Gogatishvili}$^4$ \MakeLowercase{and} {\bf Jean Michel Rakotoson}$^{5*}$}

\address{$^{1}${\small  Abdus Salam School of Mathematical Sciences -}{\small GC University -Lahore, {\bf Pakistan}}}
\email{\textcolor[rgb]{0.00,0.00,0.84}{\small quantized84@yahoo.com}}

\address{{$^{2}$ Universit\`a di Napoli "Federico II",}
{\small  via Monteoliveto, 3, I-80134 Napoli,  {\bf Italy}}
{\small and Istituto per le Applicazioni del Calcolo "Mauro Picone" }
{\small Consiglio Nazionale delle Ricerche} 
{\small via Pietro Castellino, 111 I-80131Napoli,  {\bf Italy}}}
\email{\textcolor[rgb]{0.00,0.00,0.84}{\small    fiorenza@unina.it}}

\address{$^{3}${\small Universit\`a degli Studi di Napoli "Parthenope",
via  Generale Parisi 13, 80132, Napoli, {\bf Italy}}}
\email{\textcolor[rgb]{0.00,0.00,0.84}{\small   mara.formica@uniparthenope.it}}

\address{$^{4}${\small  Institute of  Mathematics of the   Czech Academy of Sciences -
 }
 {\small 
\v Zitn\'a, 115 67 Prague 1, {\bf Czech Republic}}}
\email{\textcolor[rgb]{0.00,0.00,0.84}{\small  gogatish@math.cas.cz}}

\address{$^{5}$ 
{\small Laboratoire de Mathématiques et Applications - Université de Poitiers,}
{\small 11  Bd  Marie et Pierre Curie,T\'el\'eport 2, }
{\small 86073 Poitiers Cedex 9, {\bf France}}} 
\email{\textcolor[rgb]{0.00,0.00,0.84}{\small  rako@math.univ-poitiers.fr,   jean.michel.rakotoson@univ-poitiers.fr }}





\date{Received: xxxxxx; Revised: yyyyyy; Accepted: zzzzzz.
\newline \indent $^{*}$Corresponding author}

\begin{abstract}
 We compute the $K$-functional related to some couple of spaces as small or classical Lebesgue space or Lorentz-Marcinkiewicz spaces completing the results of \cite {FFGKR}. This computation allows to determine the interpolation space in the sense of Peetre for such couple. It happens that the result is always a $\GG$-space, since this last space covers many spaces.\\
 The motivations of such study are various, among them we wish to obtain a regularity estimate for the so called very weak solution of a linear equation in a domain $\O$ with data in the space of the integrable function with respect to the distance function to the boundary of $\O$.
\end{abstract} \maketitle

\section{\bf \large Introduction}
The present work finds its motivation in the recent results in \cite{FFR,DR,Rak}. The original question comes from an unpublished manuscript by H. Brezis (see comments in \cite{DR}) and later presented in \cite{Br2} (see also the mention made in \cite{V}) concerning the following problem :
let $f$ be given in $\LUOdpO$ ($\O$ bounded smooth open set of $\R^n$), then   H. Brezis shows the existence and uniqueness of a function $u\in \LUO$ satisfying
$$|u|_\LUO\LEQ c|f|_\LUOdpO$$
with
$$GD(\O)=\begin{cases}-\DST\int_\O u\Delta\f\dx=\int_\O f\f\dx,\quad\forall\f\in C^2_0(\OV\O),\\
\hbox{with }C^2_0(\OV\O)=\Big\{\f\in C^2(\OVO),\ \f=0\hbox{ on }\p\O\Big\}.
\end{cases}$$
Therefore, the question of the integrability of the generalized derivative $v:\p_i v=\DST\frac{\p v}{\p x_i}$ arises in a natural way and was raised already in the note by H. Brezis and developed in \cite{DR}, \cite{Rak}, \cite{rak2}. More generally, the question of the regularity of $u$ is arised, according to $f$.

In \cite{DGRT, FFR}, we have shown the following theorem:

\begin{theorem}\label{tt1}\ \\
{\it Let $\O$ be a bounded open set of class $C^2$ of $\R^n$, $|\O|=1$ and $\a\GEQ\dfrac1{n'}$ where $n'=\dfrac n{n-1},$ $ f\in L^1(\O;\d),$ with $\d(x)=\dist(x;\p\O).$ \\

Consider $u\in L^{n',\infty}(\O),$ the very weak solution (\vws) of 
\begin{equation}\label{eqeq1}
-\int_\O u\Delta\f dx=\int_\O f\f dx\qquad\forall\,\f\in C^2(\OVO),\ \f=0\ on\ \p\O.\end{equation}
Then,
\begin{enumerate}
\item if  $f\in L^1\Big(\O;\d(1+|\Log\d|)^\a\Big) $ and $\a>\dfrac1{n'}:$ 
$$u\in L^{(n',n\a- n+1}(\O)=G\G(n',1;w_\a),\ w_\a(t)=t^{-1}(1-\log t)^{\a-1-\frac1{n'}}$$
and
\begin{equation}\label{eqeq2}
||u||_{G\G(n',1;w_\a)}\LEQ K_0|f|_{L^1\left(\O;\d(1+|\Log\d|)^\a\right)}\end{equation}
\item if $f\in L^1\Big(\O;\delta\big(1+|\Log\delta|\big)^{\frac1{n'}}\Big)$ then
$$u\in L^{n'}(\O) \hbox{ and similar estimate as (\ref{eqeq2}) holds.}$$
\end{enumerate}
}
\end{theorem}
\ \\
Note that the assumption on the regularity of $\O$, needed in the proof of Theorem \ref{tt1} is necessary, for the development of the theory of very weak solutions; we stress that the estimates in this paper will be obtained following arguments valid regardless of the regularity of $\O$, which will be definitively dropped in our statements.\\

The Lorentz $\GG$-space is defined as follows :
\begin{definition}{\bf of   Generalized Gamma space with double weights (Lorentz-$\GG$)}\\
{\it Let $w_1,\ w_2$  be two weights on $(0,1)$, $m\in[1,+\infty]$, $1\LEQ p<+\infty$. We assume the following conditions:
\begin{itemize}
\item[c1)] There exists $K_{12}>0$ such that $w_2(2t)\LEQ K_{12}w_2(t)\ \forall\,t\in(0,1/2)$. The space
$L^p(0,1;w_2)$ is continuously embedded in $L^1(0,1)$.
\item[c2)] The function $\DST\int_0^tw_2(\s)d\s$ belongs to $L^{\frac mp}(0,1;w_1)$.
\end{itemize}
A generalized Gamma space with double weights is the set :
$$G\G(p,m;w_1,w_2)=\left\{v:\O\to\R\hbox{ measurable }\int_0^tv_*^p(\s)w_2(\s)d\s \hbox{ is in }L^{\frac mp}(0,1;w_1)\right\}.$$
}
\end{definition}
A similar definition has been considered in \cite{GKPS}. They were interested in the embeddings between $\GG$-spaces.
\begin{property}\ \\
{\it
Let $G\G(p,m;w_1,w_2)$ be a Generalized Gamma space with double weights and let us define for $v\in G\G(p,m;w_1,w_2)$
$$\rho(v)=\left[\int_0^1w_1(t)\left(\int_0^tv_*^p(\s)w_2(\s)d\s\right)^{\frac mp}dt\right]^{\frac1m}$$
with the obvious change for $m=+\infty$.\\
Then,
\begin{enumerate}
\item $\rho$ is a quasinorm.
\item$G\G(p,m;w_1,w_2)$ endowed  with $\rho$ is a quasi-Banach function space.
\item If $w_2=1$ $$G\G(p,m;w_1,1)=G\Gamma(p,m;w_1).$$
\end{enumerate}
}
\end{property}

\begin{ex}{\bf of weights}\\
Let $w_1(t)=(1-\Log t)^\g,\quad w_2(t)=(1-\Log t)^\b$ with $(\g,\b)\in\R^2$. Then
$$w_2\hbox{ satisfies condition c1) and $w_1$ and $w_2$  are in  }L_{exp}^{\max(\g;\b)}\big(]0,1[\big)\,.$$
\end{ex}

{\bf Question 1} The natural question is {\it how to extend of Theorem \ref{tt1} for $\a<\dfrac1{n'}$  and how to improve the estimate  when $\alpha=\dfrac1{n'}$}?\\

Since the solution of (\ref{eqeq1}) satisfies also
\begin{equation}\label{eqeq3}
|u|_{L^{n',\infty}(\O)}\LEQ K_1|f|_{L^1(\O;\d)},\end{equation}
the natural idea to obtain an estimate is to use the real interpolation method of Marcinkiewicz (see \cite{Bennett-Sharpley, Bergh-Lofstrom,  JCHC} ) to derive 
\begin{equation}\label{eqeq4}
|u|_{(L^{n',\infty},L^{(n'})_{\a,1}}\LEQ K_2|f|_{L^1(\O;\d(1+|\Log\d|)^{\a})}\qquad\hbox{for }0<\a\LEQ1.\end{equation}
Note that $L^{(n',1}=L^{(n'}$ (see below for a full definition.)\\

\ \\
{\bf Question 2} {\it How to characterize the space $\Big(L^{n',\infty}(\O),L^{(n'}(\O)\Big)_{\a,1}?$}\\

We still  have not an answer to this question. Therefore, we will provide a lower estimate for the norm of $u$ in relation (\ref{eqeq4}), a 
 particular overbound can be obtained from our work  made in \cite{FFGKR} :\\
Since $L^{n',\infty}(\O)\subset L^{n')}$, then we have
$$\Big(L^{n',\infty}(\O),L^{(n'}(\O)\Big)_{\a,1}\subset\Big(L^{n')}(\O),L^{(n'}(\O)\Big)_{\a,1}$$
and we have shown in \cite{FFGKR} the following
\begin{theorem}\label{tt2}{\bf (characterization of the interpolation between Grand and Small Lebesgue space)}\\
{\it 
$$\Big(L^{n')}(\O),L^{(n'}(\O)\Big)_{\a,1}=G\G(n';1;w_1;w_2)\hbox{ with } w_1(t)=\dfrac{(1-\Log t)^{\a-1}}t,\ w_2(t)=\dfrac1{1-\Log t}.$$
}
\end{theorem}
(see next section for the definition of $\GG$).\\

Therefore, we  have the following {\bf non optimal result} but valid for all $\alpha.$\\

\begin{proposition}\label{p1}\ \\
{\it
Let $u$ be the solution of (\ref{eqeq1}). Then,
$$||u||_{G\G(n';1;w_1;w_2)}=\int_0^1(1-\Log t)^{\a}\left(\int_0^t\dfrac{u_*^{n'}(x)dx}{1-\Log x}\right)^{\frac1{n'}}\dfrac{dt}{(1-\Log t)t}\LEQ K_4|f|_{L^1(\O;\d(1+|\Log\d|)^\a)}$$
whenever $0<\a<1.$
}
\end{proposition}
\ \\
To give a new improved statement for Proposition 1.5 namely, we will show the 
\begin{theorem}\ \\
 Let $1< p<\infty$, $0 < \theta<1$ and $1\LEQ r<\infty.$ Then
$$\|f\|_{ G\Gamma(p,r;w_1,w_2) \cap G\Gamma(\infty,r;v_1,v_2)}\lesssim \|f\|_{(L^{p,\infty}, L^{(p} )_{\theta,1}}  ,$$
where  $w_1(t)=t^{-1}(1-\ln t)^{r\theta-1}$, $w_2(t)= (1-\ln t)^{-1}, $  $v_1(t)=t^{-1}(1-\ln t)^{r\theta(1-1/p)-1}$ and $v_2(t)= t^{1/p}.$
\end{theorem}

Thanks to this last theorem, we deduce from relation (1.4) a new estimate of the solution $u$ valid also for $\alpha<\dfrac1{n'}$ and better than Proposition 1.5 in that case. 
To complete the results of \cite{FFGKR}, we shall introduce
 different results on the  interpolation spaces namely, between $(L^{n'},L^{(n'})_{\t,r}$,   $(L^{n',\infty},L^{n'})_{\t,r}$, $(L^{p),\alpha},L^{p),\beta})_{\theta,r}$, $(L^{p)},L^p)_{\t,r}$, $(L^{p,\infty},L^p)_{\t,r}$.
 It happens all of these spaces are Lorentz G-gamma spaces.  We state few of those results.
\begin{theorem}\label{t1.3}\ \\
{\it 
For $0<\t<1,\ r\in[1,+\infty[$
$$(L^{n'},L^{(n'})_{\t,r}=G(n',r;w_1;1)\hbox{\quad with $w_1(t)=t^{-1}(1-\Log t)^{r\frac\t n-1}$.}$$

}

\end{theorem}

\begin{corollary}{\bf of Theorem \ref{t1.3}}\ \\
{\it 
For $0<\t<1$, one has
$$(L^{n'},L^{(n'})_{\t,1}
=L^{(n',\t}.$$
}
\end{corollary}
 
As in \cite{FFGKR, DGRT}, the proofs of the above results rely on the computation of the $K$-functional, as for the couple $(L^{n',\infty},L^{n'})$, we will show the following 
\begin{theorem}\label{th1.4}\ \\
{\it
The $K$-functional for $(L^{n',\infty},L^{n'})$ is given by, for $t\in]0,1[$, $f\GEQ0$ in $L^{n',\infty}$
$$K_0(f;t)=t\sup\left\{\left(\int_Ef_*^{n'}(\s)d\s\right)^{\frac1{n'}};\quad t^{-n'}=\int_E\dfrac{dx}x\right\}$$
}
\end{theorem}

\begin{remark}\ \\
{\it
Setting $\DST d\nu=\dfrac{dx}x,\ |E|_\nu=\int_E d\nu,\quad f_{*,\nu}$ the decreasing rearrangement of a  nonnegative function $f$ with respect to the measure $\nu$, then we can write the preceding theorem as :
}
\end{remark}

\begin{theorem}\label{t1.4b}\ \\
{\it
The $K$-functional for the couple $(L^{p,\infty},L^p)$ is given, for $f\GEQ0$ in $L^p+L^{p,\infty},\ t>0$
$$K_0(f;t)=t\left(\int_0^{t^{-p}}\Big(\psi(s)\Big)^p_{*,\nu}(x)dx\right)^{\frac1p}.$$
Here $1\LEQ p<+\infty$, $\psi(s)=s^{\frac1p}f_*(s),\ s\in(0,1)$.
}
\end{theorem}
From this result, we can recover the following result due to Maligranda and Persson (see \cite{Maligranda-Persson} :

\begin{theorem}\label{t1.5}\ \\
{\it
Let $0<\t<1,\ 1<p<+\infty$. Then
$$(L^{p,\infty},L^p)_{\t,\frac p\t}=L^{p,\frac p\t}.$$
Here $L^{p,\frac p\t}$ is the usual Lorentz space.
}
\end{theorem}
Applying Theorem \ref{th1.4} with real interpolation method of Marcinkiewiecz, we then deduce the following partial answer for very weak solution :
\begin{proposition}\label{pf}\ \\
For $0<\a\LEQ1$, let $u$ be the solution of (\ref{eqeq1}). Then  one has  a constant $c>0$ such  that
$$
\int_0^1 t^{-\a}\sup_{\{E:|E|_\nu=t^{-n'}\}}\left(\int_Eu_*^{n'}(x)dx\right)^{\frac1{n'}}dt
\LEQ c|f|_{L^1(\O;\d(1+|\Log\d|)^{\frac\a{n'}})}. $$
\end{proposition}
\ \\
Other consequences of the above interpolation results are the interpolation inequalities, we state few of them.\\
\begin{property}\label{propri1}{\bf (Interpolation inequalities for small and grand Lebesgue spaces)}
{\it
\begin{enumerate}
\item Let $1\GEQ \a>\dfrac1{n'}$  then $\forall\, v\in L^{(n'}$
$$||v||_{L^{n',\infty}}\LEQ c||v||^{1-\a}_{L^{n')}}\,||v||^\a_{L^{(n'}}.$$
\item For any $\a\in]0,1[$, one has
$$||v||_{(L^{n',\infty},L^{(n'})_{\a,1}}\LEQ c||v||^{1-\a}_{L^{n'}}\, ||v||^\a_{L^{(n'}}\qquad\forall\,v\in L^{(n'}.$$
\end{enumerate}
}
\end{property}
\section{\bf \large Notation and Primary results}
For a measurable  function $f:\O\to\R$, we set for $t\GEQ 0$ $$D_f(t)=\hbox{measure }\Big\{ x\in\O:|f(x)|>t\Big\}$$ and 
$f_*$ the decreasing rearrangement of $|f|$, $$f_*(s)=\inf\Big\{t:D_f(t)\LEQ s\Big\} \hbox{ with }s\in\big(0,|\O|\big),\ |\O|\hbox{ is the measure of }\O,$$
 that we shall assume to be equal to 1 for simplicity.\\
 
 If $A_1$ and $A_2$ are two quantities depending on some parameters, we shall write      
 $$A_1\lesssim A_2\hbox{ if there exists $c>0$ (independent of the parameters) such that $A_1\LEQ c A_2$}$$
$$A_1\simeq A_2 \hbox{ if and only if $A_1\lesssim A_2$ and $A_2\lesssim A_1.$}$$  
We recall also the following definition of interpolation spaces.\\
Let $(X_0,||\cdot||_0),\ (X_1,||\cdot||_1)$ two Banach spaces contained continuously in a Hausdorff topological vector space (that is $(X_0,X_1)$ is a compatible couple).\\
For $g\in X_0+X_1,\ t>0$ one defines the so called $K$ functional $K(g,t;X_0,X_1)\dot=K(g,t)$ by setting
\begin{equation}\label{eq1}
K(g,t)=\inf_{g=g_0+g_1}\big(||g_0||_0+t||g_1||_1\big).
\end{equation}
For $0\LEQ\theta\LEQ1,\ 1\LEQ p\LEQ+\infty,\ \a\in\R$ we shall consider
$$(X_0,X_1)_{\theta,p;\a}=\Big\{g\in X_0+X_1,\ ||g||_{\theta,p;\a}=||t^{-\theta-\frac1p}\big(1-\Log t\big)^\a K(g,t)||_{L^p(0,1)}\hbox{ is finite}
\Big\}.$$
Here $||\cdot||_V$ denotes the norm in a Banach space $V$. The weighted Lebesgue space $L^p(0,1;\omega)$, $0< p\LEQ+\infty$ is endowed with the usual norm or quasi norm, where $\omega$ is a weight function on $(0,1)$.\\
Our definition of the  interpolation space is different from the usual one (see \cite{Bennett-Sharpley, Tartar}) since we restrict the norms on the interval $(0,1)$.\\ If  we consider ordered couple, i.e.  $X_1\hookrightarrow X_0$ and $\a=0,$ 
 $$(X_0,X_1)_{\theta,p;0}=(X_0,X_1)_{\theta,p}$$ is the interpolation space as it is defined by  J. Peetre (see \cite{Bennett-Sharpley, Tartar, Bergh-Lofstrom}). \\

\subsection{Some remarkable $\GG$-spaces}\ \\
In this paragraph, we want to prove among other that $\GG$-spaces cover many well-known spaces.
\begin{proposition}\label{p20}\ \\
{\it 
Consider the classical Lorentz space $\DST\Lambda^p(w_2)$. Then it is equal to the set \\ $\DST\Big\{f:\O\to\R\hbox{ measurable : }\left(\int_0^1f_*^p(\s)w_2(\s)d\s)\right)^{\frac1p}=||f||_{\Lambda^p(w_2)}<+\infty\Big\}.$\\
If $w_1$ and $w_2$ are integrable weights on $(0,1)$ and $w_2$ satisfies c1) then
$$\GG(p,m;w_1,w_2)=\Lambda^p(w_2).$$
}
\end{proposition}

{\bf Proof}\\
If $v\in\Lambda^p(w_2)$ then $\DST\rho(v)\LEQ||v||_{\Lambda^p(w_2)}\left[\int_0^1w_1(t)dt\right]^\frac1m<+\infty$.\\
Conversely, let $v$ be such that $\rho(v)<+\infty$. We have for some $a>0$, $\DST\int_a^1 w_1(t)dt>~0$. Then for all $t\GEQ a$
$$\left(\int_0^af_*^p(\s)w_2(\s)d\s)\right)^{\frac mp}\LEQ \left(\int_0^t f_*^p(\s)w_2(\s)d\s\right)^{\frac mp},$$
from which we derive after multiplying  by $w_1(t)$ and integrating from $a$ to 1,  
\begin{equation}\label{eq51}
\left(\int_0^af_*^p(\s)w_2(\s)d\s\right)^{\frac1p}\LEQ\dfrac{\rho(v)}{\left[\int_a^1 w_1(t)dt\right]^{\frac1m}}\lesssim\rho(v)<+\infty.
\end{equation}
Between $(a,1)$, we have :
\begin{equation}\label{e60}\left(\int_a^1 f_*^p(\s)w_2(\s)d\s\right)^{\frac1p}\LEQ f_*(a)||w_2||^{\frac1p}_{L^1}\lesssim\int_0^af_*(\s)d\s=\int_0^1f_*(\s)\chi_{[0,a]}(\s)d\s.\end{equation}
The condition c1) implies
\begin{equation}\label{e61}
\int_0^af_*(\s)d\s\lesssim\left(\int_0^1(f_*\chi_{[0,a]})^p(\s)w_2(\s)d\s\right)^{\frac1p}.\end{equation}
So that relations (\ref{eq51}) to (\ref{e61}) imply
\begin{equation}\label{eq52}
\left(\int_a^1f_*^p(\s)w_2(\s)d\s\right)^{\frac1p}\lesssim\left(\int_0^af_*^p(\s)w_2(\s)d\s\right)^{\frac1p}\lesssim\rho(v)<+\infty.
\end{equation}
This shows $$||f||_{\Lambda^p(w_2)}\lesssim\rho(v).$$ 
\ \HF
Next we want to focus in a special case :
\begin{proposition}\label{p21}\ \\
{\it
Assume that $w_1(t)=t^{-1}(1-\Log t)^\g,$ $ w_2(t)=(1-\Log t)^\b,$ $(\g,\b)\in\R^2,$ $ m\in[1,+\infty[,$ $ p\in[1,+\infty[$.
\begin{enumerate}
\item  If $\g<-1$ then $\GG(p,m;w_1,w_2)=\Lambda^p(w_2)$.
\item If $\g>-1$ and $\g+\b\dfrac mp+1\GEQ0$ then
$$\GG(p,m;w_1,w_2)=\GG(p,m;\OV{w_1},1),\ \OV{w_1}(t)=t^{-1}(1-\Log t)^{\g+\b\frac mp}.$$
\end{enumerate}
}
\end{proposition}
{\bf Proof}\ \\
For the first statement, we observe that if $\g+1<0$, $\DST\int_0^1(1-\Log t)^\g\dfrac{dt }t$ is finite. Then applying Proposition \ref{p20} we derive the first result.\\
For the case $\g+1>0$, we shall need the following lemma whose proof is in~\cite{FFGKR}:
\begin{lemma}\label{l60}\ \\
{\it 
Let $t_k=2^{1-2^k},\ k\in\N,\ \lambda>0,\ q>0,\ H$ a nonnegative locally integrable function on $(0,1)$ satisfying
$$\int_0^1H(x)dx\lesssim\int_0^{\frac12}H(x)dx.$$ Then
\begin{enumerate}
\item $2^k\approx 1-\Log x$, $x\in[ t_{k+1},t_k]$.
\item
 \begin{eqnarray*}
\DST\int_0^1\left[(1-\Log t)^\lambda\int_0^t H(x)dx\right]^q\dfrac{dt}{(1-\Log t)t}
&\approx&\sum_{k\in\N}\left(\int_0^{t_k} H(x)dx\right)^q2^{\lambda kq}\\
&\approx&\sum_{k\in\N}\left(2^{\lambda k}\int_{t_{k+1}}^{t_k}H(x)dx\right)^q.
\end{eqnarray*}
\end{enumerate}
}
\end{lemma}
 
We shall apply this Lemma with $\DST H(x)=f_*^p(x)(1-\Log x)^\b$. We have 
$\DST\int_0^1H(x)dx\lesssim\int_0^{\frac12}H(x)dx$ since $f_*^p$ is decreasing and
$\DST\int_0^1(1-\Log t)^\g dt< +\infty\quad\forall\,\g\in\R$. Indeed
\begin{eqnarray*}
\left(\int_{\frac12}^1H(x)dx\right)^{\frac1p}
&\lesssim &f_*\left(\dfrac12\right)\lesssim\int_0^{\frac12}f_*(t)dt\\&\LEQ&\left(\int_0^{\frac12}f_*^p(t)(1-\Log t)^\b dt\right)^{\frac1p}\cdot\left(\int_0^{\frac12}(1-\Log t)^{-\b\frac{p'}p}dt\right)^{\frac1{p'}}\\
&\lesssim&\left(\int_0^{\frac12}H(x)dx\right)^{\frac1p}.\end{eqnarray*}
\\
Applying statement 2. of  this Lemma \ref{l60}, we derive
\begin{eqnarray}
\rho^m(f)&=&\int_0^1(1-\Log t)^\g\left(\int_0^tH(x)dx\right)^{\frac mp}\dfrac{dt}t\NN\\
&=&\int_0^1\left[(1-\Log t)^{(\g+1)\frac pm}\int_0^tH(x)dx\right]^{\frac mp}\dfrac{dt}{(1-\Log t)t}\quad if\ \g+1>0\NN\\
&\approx&\sum_{k\in\N}\left(2^{\lambda k}\int_{t_{k+1}}^{t_k}H(x)dx\right)^q\hbox{ with }\lambda=(\g+1)\dfrac pm,\ q=\dfrac mp\NN\\
\end{eqnarray} 
\begin{eqnarray}
&\approx&\sum_{k\in\N}\left(2^{(\lambda+\b)k}\int_{t_{k+1}}^{t_k}f_*^p(x)dx\right)^q\NN\\
&\approx&\int_0^1\left[(1-\Log t)^{\lambda+\b}\int_0^tf_*^p(x)dx\right]^q\dfrac{dt}{(1-\Log t)t}\NN\\
&\approx&\int_0^1(1-\Log t)^{(\lambda+\b)q-1}\left(\int_0^t f^p_*(x)dx\right)^q\dfrac{dt}t.\label{e70}
\end{eqnarray} 

If $\g>-1,\ \g+\b\dfrac mp+1\GEQ0$, then the equality comes from the definition of $\GG(p,m;\OV{w_1})$.\\ 
This ends of the proof  of Proposition \ref{p21}\HF

\begin{lemma}\label{cvL1} \ \\
Assume that $w_1(t)=t^{-1} (1- \ln t)^\gamma,$ $w_2(t)= (1- \ln t)^\beta,$ $(\gamma, \beta)\in \mathbb{R}^2$, $m \in [1, \infty[,$ $ p \in [1, \infty[.$ If $\gamma> -1$ and $\gamma + \beta\frac{m}{p}+1<0$, then

$$\|f\|^m_{G\Gamma(p,m;w_1,w_2)}\approx\int_{0}^{1}(1-\ln t)^{\gamma+ \beta\frac{m}{p}}\left(\int_{t}^{1}f_\ast(x)^pdx\right)^{m/p}\frac{dt}{t}. $$
\end{lemma}
\begin{proof}
Put
$$I=\int_{0}^{1}(1-\ln t)^{\gamma+ \beta\frac{m}{p}}\left(\int_{t}^{1}f_\ast(x)^pdx\right)^{m/p}\frac{dt}{t}. $$
 Let $t_k=2^{1-2^k},$ $k\in \mathbb{N}.$ Since $\gamma + \beta\frac{m}{p}+1<0$, we can apply  the second assertion of Lemma $6.3$ in \cite{FFGKR} to obtain
$$ I\approx \sum\limits_{k \in \mathbb{N}}\left(\int_{t_{k+1}}^{1}f_\ast(x)^pdx\right)^{m/p} 2^{(\gamma + \beta\frac{m}{p}+1)k}, $$
then  using the  second assertion of Lemma $6.1$ in \cite{FFGKR} gives
$$ I\approx \sum\limits_{k \in \mathbb{N}}\left(\int_{t_{k+1}}^{t_k}f_\ast(x)^pdx\right)^{m/p} 2^{(\gamma + \beta\frac{m}{p}+1)k}, $$
by first assertion in Lemma $2.3$, we get
$$ I\approx \sum\limits_{k \in \mathbb{N}}\left(\int_{t_{k+1}}^{t_k}(1-\ln x)^\beta f_\ast(x)^pdx\right)^{m/p} 2^{(\gamma +1)k}, $$
since $\gamma> -1$ , we can apply the first assertion of Lemma $6.1$ in \cite{FFGKR} to obtain
$$ I\approx \sum\limits_{k \in \mathbb{N}}\left(\int_{0}^{t_k}(1-\ln x)^\beta f_\ast(x)^pdx\right)^{m/p} 2^{(\gamma +1)k}, $$
finally an application of the second assertion in Lemma $2.3$ yields
$$I\approx \|f\|^m_{G\Gamma(p,m;w_1,w_2)},$$
which completes the proof.
\end{proof}

We shall need in particular the   Corollary \ref{c2}, consequence of relation (\ref{e70}) and the following 
\begin{definition}{\bf of the small Lebesgue space \cite{IS, dff}}\\
{\it
The small Lebesgue space associated to the parameters $p\in]1,+\infty[$ and $\theta>0$ is the set\\
$L^{(p,\theta}(\O)=$$$\Big\{ f:\O\to\R\hbox{ measurable : } ||f||_{(p,\theta}=\int_0^1(1-\Log t)^{-\frac\theta p+\theta-1}\left(\int_0^tf^p_*(\s)d\s\right)^{1/p}\dfrac{dt}t<+\infty\Big\}.$$
}
\end{definition}
\begin{definition}{\bf of the grand Lebesgue space \cite{IS, dff}}\\
{\it
The grand Lebesgue space is the  associate space of the small Lebesgue space, with the parameters $p\in]1,+\infty[$ and $\theta>0$ is the set\\
$L^{p),\theta}(\O)=$$$\Big\{ f:\O\to\R\hbox{ measurable : } ||f||_{p),\theta}=\sup_{0<t<1}(1-\Log t)^{-\frac\theta p}\left(\int_t^1f^p_*(\s)d\s\right)^{1/p}\dfrac{dt}t<+\infty\Big\}.$$
}
\end{definition}

\begin{corollary}\label{c2}{\bf of Proposition \ref{p21}}\ \\
{\it
If $m=1, \g+1+\dfrac \beta p>0$, $\gamma>-1$ and $\beta\in\R$, the functions $w_i,\ i=1,2$ as in Proposition \ref{p21} then
$$\GG(p,1;w_1,w_2)=L^{(p,\theta},\ \theta=p'\left(\g+1+\dfrac\b p\right).$$
}
\end{corollary}
 
\section{\bf \large Some   $K$-functional computations and the  associated interpolation spaces}

{\it   
}
\subsection{The case of the couple $(L^{n'},L^{(n'})$}

\begin{theorem}\label{311}\ \\
{\it
Let $\f(t)=e^{1-\frac1{t^n}},\ 0<t\LEQ1$. Then
$$K(f,t;L^{n'},L^{(n'})\approx t\int_{\f(t)}^1(1-\Log\s)^{-\frac1{n'}}\left(\int_0^\s f_*^{n'}(x)dx\right)^{\frac1{n'}}\dfrac{d\s}\s\dot=K^2(t)$$
for all $f\in L^{n'}+L^{(n'}$.
}
\end{theorem}
 
 \ 
 \pr
First , let us show:
$$K^2(t)\lesssim K(f,t;L^{n'},L^{(n'}).$$

Let $f=g+h\in L^{n'}+L^{(n'}$. Then, for all $x$,  $f_*(x)\LEQ g_*\left(\dfrac x2\right)+h_*\left(\dfrac x2\right)$. Therefore, we have
\begin{eqnarray*}
K^2(t)&\LEQ&||g||_{L^{n'}}t\int^1_{\f(t)}(1-\Log\s)^{-\frac1{n'}}\dfrac{ d\s}\s+t||h||_{L^{(n'}}\\
&\lesssim&||g||_{L^{n'}}+t||h||_{L^{(n'}}
\end{eqnarray*}
Taking the infinimum, one derives
\begin{equation}\label{eq3111}
K^2(t)\lesssim K(f,t;L^{n'},L^{(n'}).
\end{equation}

For the converse, we adopt the same decomposition as in \cite{FFGKR}
\begin{equation}\label{3.2019}
g=\left(|f|-f_*\big(\f(t)\big)\right)_+,\quad h=f-g.
\end{equation}
Then 
$$f_*=h_*+g_*,\qquad g_*=\Big(f_*-f_*\big(\f(t)\big)\Big)_+,$$
$$h_*=f_*\big(\f(t)\big)\chi_{[0,\f(t)]}+f_*(s)\chi_{[\f(t),1]}$$
\begin{eqnarray}\label{eq3.2}
||g||_{L^{n'}}&\lesssim&\left[\int_0^{\frac{\f(t)}2}\Big(f_*(s)-f_*\big(\f(t)\big)_+^{n'}ds\right]^{\frac1{n'}}
\lesssim\left[\int_0^{\frac{\f(t)}2}f^{n'}_*(s)ds\right]^{\frac1{n'}}\nonumber\\
&=&||f||_{L^{n'}(0,\frac{\f(t)}2)}\lesssim
\dfrac{K^2(t)}{\DST\int_{\f(t)/2}^1(1-\Log\s)^{-\frac1{n'}}\dfrac{ d\s}\s}\lesssim K^2(t).
\end{eqnarray}
As in \cite{FFGKR}, we have 
\begin{eqnarray}
t||h||_{L^{(n'}}&\LEQ&t\left(\int_0^{\f(t)}(1-\Log s)^{-1/n'}s^{1/n'}\dfrac{ds}s\right)f_*\big(\f(t)\big)\NN\\
&&+t\left(\int_{\f(t)}^1(1-\Log s)^{-1/n'}\dfrac{ds}s\right)\f(t)^{1/n'}f_*\big(\f(t)\big)\NN\\
&&+t\int_{\f(t)}^1(1-\Log s)^{-1/n'}\left(\int_{\f(t)}^sf_*^{n'}(x)dx\right)^{1/n'}\dfrac{ds}s\label{eq8}\\
&=&I_1+I_2+I_3.\NN
\end{eqnarray}
Since
$$\int_0^{\f(t)}s^{1/n'}(1-\Log s)^{-1/n'}\dfrac{ds}s\lesssim\f(t)^{1/n'}\Big(1-\Log\f(t)\Big)^{-1/n'},$$
we obtain for the first term $I_1$
\begin{equation}\label{eq9}
I_1\lesssim t\Big(1-\Log\f(t)\Big)^{-1/n'}\f(t)^{1/n'}f_*\big(\f(t)\big)\lesssim  K^2(t)
\end{equation}
\begin{equation}\label{eq10}
I_2\lesssim t^{1/n'}\f(t)^{1/n'}f_*\Big(\f(t)\Big)\lesssim t^{1/n'}\sup_{0<s<\f(t)}s^{1/n'}f_*(s)\lesssim K^2(t)
\end{equation}
and
\begin{equation}\label{eq11}
I_3\lesssim K^2(t)\end{equation}
with relations (\ref{eq9}) to (\ref{eq11}), we derive
\begin{equation}\label{eq12}
t||h||_{L^{(n'}}\LEQ I_1+I_2+I_3\lesssim K^2(t).
\end{equation}
Thus relations (\ref{eq3.2}) and (\ref{eq12}) infer :
\begin{equation}\label{eq13}
||g||_{L^{n'}}+t||h||_{L^{(n'}}\lesssim K^2(t).
\end{equation}
Thus
\begin{equation}\label{eq14}
K(f,t;L^{n'},L^{(n'})\lesssim K^2(t).
\end{equation}
The combination of the above relations (\ref{eq14}), (\ref{eq3111}) gives Theorem \ref{311}.\HF
\begin{corollary}{\bf Theorem \ref{311}}\label{ct}\ \\
{\it
One has, for $r\in[1,+\infty[$,\ \  $0<\t<1$,
$$||f||^r_{(L^{n'},L^{(n'})_{\t,r}}\approx\int_0^1(1-\Log x)^{\frac{\t r}n}\left(\int_0^x
f^{n'}_*(s) ds\right)^{\frac r{n'}}\dfrac{dx}{x(1-\Log x)}.$$
}
\end{corollary}
\pr
\def\lnp{{L^{n'}}}
\def\lpnp{{L^{(n'}}}
\def\up{{\frac1p}}
One has for $f\in \lnp+\lpnp,\ 1\LEQ r+\infty$
\begin{equation}\label{eq3112}
||f||^r_{(\lnp,\lpnp)_{\t,r}}=\int_0^1[t^{-\t}K(f,t)]^r\dfrac{dt}t
\end{equation}
Using Theorem \ref{311} and making a change of variable $x=\f(t)$ that is \\$t=(1-\Log x)^{-\frac1n}$, one derives from relation (\ref{eq3112})
$$||f||^r_{(\lnp,\lpnp)_{\t,r}}\approx J_f$$
$$J_f=\int_0^1\left[(1-\Log x)^{\frac{\t-1}n}\int_x^1(1-\Log \s)^{-\frac1{n'}}\left(\int_0^\s f_*^{n'}(x)dx\right)^{\frac1{n'}}\dfrac{d\s}\s\right]^r\dfrac {dx}{x(1-\Log x)}.$$
Applying Hardy's inequality (taking into account that $\theta<1$), we have

$$J_f\lesssim\int_0^1\left[(1-\Log x)^{\frac{\t}n}\left(\int_0^x f^{n'}_*(\s)d\s\right)^{\frac1{n'}}\right]^r\dfrac{dx}{x(1-\Log x)}=\WT J_f.$$

For the converse, since we have for all $x>0$
$$\int_x^1(1-\Log\s)^{-\frac1{n'}}\left(\int_0^\s f_*^{n'}(s)ds\right)^{-\frac1{n'}}\dfrac{d\s}\s\GEQ \left(\int_0^xf^{n'}_*\right)^{\frac1{n'}}
(1-\Log x)^{-\frac1{n'}}|\Log x|,$$

we then have
\begin{equation}\label{eq3113}
J_f\GEQ \int_0^1\left[(1-\Log x)^{\frac\t n-1-\frac1r}|\Log x|\left(\int_0^xf_*^{n'}(s)ds\right)^{\frac1{n'}}\right]^r\dfrac{dx} x.
\end{equation}

From this relation we deduce
\begin{equation}\label{eq3114}
\WT J_f\lesssim J_f+ \int_0^1\left[(1-\Log x)^{\frac\t n-1-\frac1r}\left(\int_0^xf_*^{n'}(s)ds\right)^{\frac1{n'}}\right]^r\dfrac{dx} x\dot=
J_f+I_r,
\end{equation}
while to estimate the last integral, one has
$$I_r\LEQ||f||^r_\lnp\int_0^1(1-\Log x)^{(\frac\t n)r-1}\dfrac{dx}x\LEQ c||f||^r_\lnp.$$
Since $(\lnp,\lpnp)_{\t,r}$ is continuously embedded in  $\lnp$, we then have
\begin{equation}\label{eq3115}
I_r\LEQ c||f||^r_{(\lnp,\lpnp)_{\t,r}}.
\end{equation}
Thus, we derive
$$\WT J_f\lesssim J_f+I_r\lesssim ||f||^r_{(\lnp,\lpnp)_{\t,r}}.$$
\ \ \HF

{\bf Proof of Theorem \ref{t1.3}}\\
We derive it from Corollary \ref{ct} of Theorem \ref{311}.
\ \ \HF

\subsection{Interpolation between grand and classical Lebesgue spaces in the critical case}

\begin{lemma}\label{L1gc}\ \\
 Let $1<p<\infty$, and let $f\in L^{p)}.$ Then, for all $0<t<1,$
$$
K(f,t;L^{p)},L^p)\approx \sup_{0<s<\varphi(t)}(1-\ln s)^{-1/p}\left(\int_{s}^{1}f_\ast(x)^pdx\right)^{1/p},
$$
where  $\varphi(t)=e^{1-\frac{1}{t^p}}.$
\end{lemma}
\begin{proof} Fix $f\in L^{p)}$ and $0<t<1.$ Set
$$K_1(t)=\sup_{0<s<\varphi(t)}(1-\ln s)^{-1/p}\left(\int_{s}^{1}f_\ast(x)^pdx\right)^{1/p}.$$
First we show that
\begin{equation}\label{e1gcL1}
K_1(t)\lesssim K(f,t;L^{p)},L^p).
\end{equation}
Let $f=g+h$ be an arbitrary decomposition with $g \in L^{p)}$ and $h \in L^{p}.$ Using the elementary inequality $f_\ast(x)\LEQ g_\ast(x/2)+h_\ast(x/2) $, we derive
\begin{eqnarray*}
K_1(t)&\lesssim& \|g\|_{L^{p)}}+ \|h\|_{L^p}\sup_{0<s<\varphi(t)}(1-\ln s)^{-1/p}\\
&=&\|g\|_{L^{p)}}+ t\|h\|_{L^p},
\end{eqnarray*}
from which  (\ref{e1gcL1}) follows. Next we establish the converse estimate
\begin{equation}\label{e2gcL2}
 K(f,t;L^{p)},L^p)\lesssim K_1(t).
\end{equation}
To this end, we take the same particular decomposition  $f=g+h$ as in Theorem 3.1 relation (\ref{3.2019}).  Clearly,
\begin{equation}\label{e2gcL3}
 \|g\|_{L^{p)}}\LEQ K_1(t).
 \end{equation}
Next we note that
\begin{eqnarray*}
t\|h\|_{L^p}&=&tf_\ast(\varphi(t))[\varphi(t)]^{1/p}+ t\left(\int_{\varphi(t)}^{1}f_\ast(x)^pdx\right)^{1/p}\\
&\approx& \left(1-\ln \frac{\varphi(t)}{2}\right)^{-1/p}f_\ast(\varphi(t))[\varphi(t)]^{1/p}\\
&&+ \left(\sup\limits_{0<s<\phi(t)}(1-\ln s)^{-1/p}\right)\left(\int_{\varphi(t)}^{1}f_\ast(x)^pdx\right)^{1/p},\\
&\lesssim& \left(1-\ln \frac{\varphi(t)}{2}\right)^{-1/p}\left(\int_{\frac{\varphi(t)}{2}}^{\varphi(t)}f_\ast(x)dx\right)^{1/p}\\
&&+ \sup\limits_{0<s<\phi(t)}(1-\ln s)^{-1/p}\left(\int_{s}^{1}f_\ast(x)^pdx\right)^{1/p},\\
\end{eqnarray*}
which gives
\begin{equation}\label{e2gcL4}
 t\|h\|_{L^p}\LEQ K_1(t).
 \end{equation}
Now (\ref{e2gcL2})  follows from (\ref{e2gcL3}) and (\ref{e2gcL4}). The proof is complete.
\end{proof}

\begin{theorem}\ \\
Let $1<p<\infty$, $0<\theta<1$, and  $1\LEQ r <\infty.$ Then

$$ (L^{p)},L^p)_{\theta,r}= G\Gamma(p,r;w_1,w_2),$$
where   $w_1(t)= t^{-1}(1-\ln t)^{r\theta/p-1}$ and $w_2(t)= (1-\ln t)^{-1}.$
\end{theorem}
\begin{proof} \ \\Let $f\in L^{p)}.$ Then, using Lemma \ref{L1gc}, we get at

\begin{equation}\label{e1gcT1}
\|f\|^r_{(L^{p)},L^p)_{\theta,r}}\approx \int_{0}^{1}(1-\ln t)^{\frac{r\theta}{p}-1}\left[\sup_{0<s<t}\psi(s)\right]^{r}\frac{dt}{t},
\end{equation}
where
$$
\psi(s)=(1-\ln s)^{-1/p}\left(\int_{s}^{1}f_\ast(x)^pdx\right)^{1/p}.
$$
Now,  in view of Lemma \ref{cvL1} (applied with $\gamma=\frac{r \theta}{p}-1$, $\beta=-1$, $m=r$), it is sufficient to establish that
\begin{equation}\label{e2gcT1}
\|f\|^r_{(L^{p)},L^p)_{\theta,r}}\approx \int_{0}^{1}(1-\ln t)^{r(\frac{\theta-1}{p})-1}\left(\int_{t}^{1}f_\ast(x)^pdx\right)^{r/p}\frac{dt}{t}.
\end{equation}
The estimate  $``\gtrsim"$ in (\ref{e2gcT1})  follows trivially from (\ref{e1gcT1}), while for the converse estimate we infer from Bennet-Rudnick Lemma (\cite{Bennett-Rudnick} Lemma $6.1$) that
$$
\sup_{0<s<t}\psi(s)\lesssim \int_{0}^{t}(1-\ln s)^{-1}\psi(s)\frac{ds}{s},
$$
which, combined with (\ref{e1gcT1}), gives
\begin{equation*}
\|f\|^r_{(L^{p)},L^p)_{\theta,r}}\lesssim \int_{0}^{1}\left[(1-\ln t)^{\frac{\theta}{p}-\frac{1}{r}}\int_{0}^{t}(1-\ln s)^{-1}\psi(s)\frac{ds}{s}\right]^{r}\frac{dt}{t},
\end{equation*}
from which follows the desired estimate $``\lesssim"$ in (\ref{e2gcT1}) by Hardy inequality \cite{Bennett-Rudnick} Theorem $6.5$. The proof is complete.
\end{proof}

\subsection{Interpolation between grand  Lebesgue spaces in the critical case}
\begin{lemma}\label{L1gg}\ \\
Let  $1<p<\infty$ and $0<\beta < \alpha.$  Let $f\in L^{p),\alpha}$.  Then, for all $0<t<1,$
\begin{eqnarray*}
K(f,t;L^{p),\alpha},L^{p),\beta})&\approx& \sup_{0<s<\varphi(t)}(1-\ln s)^{-\frac{\alpha}{p}}\left(\int_{s}^{\varphi(t)}f_\ast(x)^p dx\right)^{1/p}\\
&&+t \sup\limits_{\varphi(t)<s<1}(1-\ln s)^{-\frac{\beta}{p}}\left(\int_{s}^{1}f_\ast(x)^pdx\right)^{1/p},
\end{eqnarray*}
where  $\varphi(t)=e^{1-t^\frac{p}{\beta-\alpha}}.$
\end{lemma}
\begin{proof}\ \\ Fix $f\in L^{p,\infty}$ and $0<t<1.$ Set
$$ K_1(t)=\sup_{0<s<\varphi(t)}(1-\ln s)^{-\frac{\alpha}{p}}\left(\int_{s}^{\varphi(t)}f_\ast(x)^p dx\right)^{1/p},$$
and
$$K_2(t)=t \sup\limits_{\varphi(t)<s<1}(1-\ln s)^{-\frac{\beta}{p}}\left(\int_{s}^{1}f_\ast(x)^pdx\right)^{1/p}.$$
First we show that
\begin{equation}\label{e1L1gg}
K_1(t)+K_2(t)\lesssim K(f,t;L^{p),\alpha},L^{p),\beta}).
\end{equation}
Let $f=g+h$ be an arbitrary decomposition with $g \in L^{\alpha),p}$ and $h \in L^{\beta),p}.$ Using the elementary inequality $f_\ast(x)\LEQ g_\ast(x/2)+h_\ast(x/2) $, we derive
\begin{eqnarray*}
K_1(t)&\lesssim& \|g\|_{L^{p),\alpha}}+ \|h\|_{L^{p),\beta}}\sup_{0<s<\varphi(t)}(1-\ln s)^{\frac{\beta-\alpha}{p}}\\
&=&\|g\|_{L^{p),\alpha}}+ t\|h\|_{L^{p),\beta}},
\end{eqnarray*}
and
\begin{eqnarray*}
K_2(t)&\lesssim& t\|g\|_{L^{p),\alpha}}\sup_{\varphi(t)< s<1}(1-\ln s)^{\frac{\alpha-\beta}{p}}+ t\|h\|_{L^{p),\beta}}\\
&=&\|g\|_{L^{p),\alpha}}+ t\|h\|_{L^{p),\beta}}.
\end{eqnarray*}
Thus, we get
$$K_1(t)+K_2(t)\lesssim \|g\|_{L^{p),\alpha}}+ t\|h\|_{L^{p),\beta}},$$
from which  (\ref{e1L1gg}) follows.

It remains  to establish the converse estimate
\begin{equation}\label{e2L1gg}
 K(f,t;L^{p),\alpha},L^{p),\beta})\lesssim K_1(t)+K_2(t).
\end{equation}
 We again take the same particular decomposition  $f=g+h$ as in Theorem 3.1. (relation (\ref{3.2019}). It is easy to check that
 \begin{equation}\label{e3L1gg}
 \|g\|_{L^{p),\alpha}}\lesssim K_1(t).
\end{equation}
Next we observe  that
\begin{equation}\label{e4L1gg}
t\|h\|_{L^{p),\beta}}= J_1(t)+J_2(t)+K_2(t),
\end{equation}
where
$$J_1(t)=t f_\ast(\varphi(t))\sup\limits_{0<s<\varphi(t)}(1-\ln s)^{-\frac{\beta}{p}}\left( \int_{s}^{\varphi(t)}dx\right)^{1/p},$$
and
$$J_2(t)=t \left( \int_{\varphi(t)}^{1}f_\ast(x)^pdx\right)^{1/p} \sup\limits_{0<s<\varphi(t)}(1-\ln s)^{-\frac{\beta}{p}}.$$
Since $$\sup\limits_{0<s<\varphi(t)}(1-\ln s)^{-\frac{\beta}{p}}=(1-\ln \varphi(t))^{-\frac{\beta}{p}},$$
we have
\begin{equation}\label{e5L1gg}
J_2(t)\LEQ K_2(t).
\end{equation}
Next we show that
\begin{equation}\label{e6L1gg}
J_1(t)\lesssim K_1(t).
\end{equation}
We have
\begin{eqnarray*}
J_1(t)&\LEQ& t f_\ast(\varphi(t))\left( \int_{0}^{\varphi(t)}dx\right)^{1/p}\sup\limits_{0<s<\varphi(t)}(1-\ln s)^{-\frac{\beta}{p}}\\
&=& t^\frac{\alpha}{\alpha-\beta}f_\ast(\varphi(t))\left[\varphi(t) \right]^{1/p}\\
&\lesssim& (1-\ln \varphi(t) )^{-\frac{\alpha}{p}}\left(\int_{\frac{\varphi(t)}{2}}^{\varphi(t)}f_\ast(x)^p dx\right)^{1/p}\\
\end{eqnarray*}
from which follows (\ref{e6L1gg}). Altogether from the relations (\ref{e3L1gg})-(\ref{e6L1gg}), we get (\ref{e2L1gg}). The proof is complete.
\end{proof}
\begin{theorem}\ \\
Let  $1<p<\infty$, $0<\beta < \alpha,$  $0<\theta<1$, and  $1\LEQ r <\infty.$ Then

$$(L^{p),\alpha},L^{p),\beta})_{\theta,r}= G\Gamma(p,r;w_1,w_2),$$
where   $w_1(t)= t^{-1}(1-\ln t)^{\frac{r\theta}{p}(\alpha-\beta)-1}$ and $w_2(t)= (1-\ln t)^{-\alpha}.$
\end{theorem}
\begin{proof}\ \\ Let $X=(L^{p),\alpha},L^{p),\beta})_{\theta,r}$ and take $f\in L^{p),\alpha}. $ Then
$$\|f\|^r_X\approx I_1+I_2,$$
where
$$I_1= \int_{0}^{1} (1-\ln t)^{\frac{r\theta(\alpha-\beta)}{p}-1}\left[\sup\limits_{0<s<t}(1-\ln s)^{-\frac{\alpha}{p}} \left(\int_{s}^{t}f_\ast(x)^pdx\right)^{1/p} \right]^r\frac{dt}{t}, $$
and
$$I_2= \int_{0}^{1} (1-\ln t)^{\frac{r(\theta-1)(\alpha-\beta)}{p}-1}\left[\sup\limits_{t<s<1}(1-\ln s)^{-\frac{\beta}{p}} \left(\int_{s}^{1}f_\ast(x)^pdx\right)^{1/p} \right]^r\frac{dt}{t}. $$
Put
$$I_3=\int_{0}^{1} (1-\ln t)^{\frac{r\theta(\alpha-\beta)}{p}-\frac{r\alpha}{p} -1} \left(\int_{t}^{1}f_\ast(x)^pdx\right)^{r/p}\frac{dt}{t}.$$
 In view of Lemma \ref{cvL1} (applied with $\gamma= r\theta(\alpha-\beta)/p-1$, $\beta=-\alpha$, $m=r$), it is sufficient to show that
\begin{equation*}
I_1+I_2\approx I_3.
\end{equation*}
Clearly,  $I_3 \lesssim I_2$. Thus, it remains to establish that $I_1 \lesssim I_3$  and $I_2 \lesssim I_3.$ Now
$$I_1\LEQ \int_{0}^{1} (1-\ln t)^{\frac{r\theta(\alpha-\beta)}{p}-1}\left[\sup\limits_{0<s<t}(1-\ln s)^{-\frac{\alpha}{p}} \left(\int_{s}^{1}f_\ast(x)^pdx\right)^{1/p} \right]^r\frac{dt}{t}, $$
by Bennet-Rudnick Lemma (\cite{Bennett-Rudnick} Lemma $6.1$), we get
$$I_1\lesssim\int_{0}^{1} (1-\ln t)^{\frac{r\theta(\alpha-\beta)}{p}-1}\left[\int_{0}^{t}(1-\ln s)^{-\frac{\alpha}{p}-1} \left(\int_{s}^{1}f_\ast(x)^pdx\right)^{1/p}\frac{ds}{s} \right]^r\frac{dt}{t}, $$
now applying Hardy inequality \cite{Bennett-Rudnick} Theorem $6.5$, we obtain $I_1\lesssim I_3.$  Next we again   make use of  Bennet-Rudnick Lemma (\cite{Bennett-Rudnick} Lemma $6.1$) to derive
 
$$I_2\lesssim \int_{0}^{1} (1-\ln t)^{\frac{r(\theta-1)(\alpha-\beta)}{p}-1}\left[ \int_{t}^{1}(1-\ln s)^{-\frac{\beta}{p}-1} \left(\int_{s}^{1}f_\ast(x)^pdx\right)^{1/p}\frac{ds}{s} \right]^r\frac{dt}{t}, $$
from which follows  $I_2\lesssim I_3$ by Hardy inequality \cite{Bennett-Rudnick} Theorem $6.5$. The proof is complete.
\end{proof}

\subsection{\bf \large The $K$-functional for the couple $(L^{p,\infty},L^p),\ 1<p<+\infty$}
\begin{theorem}\label{ttt1}\ \\
{\it
For a measurable set $E\subset[0,1]$, we denote $|E|_\nu=\DST\int_E \dfrac{dx}x$ and for $f\in L^{p,\infty}+L^p$, $1<p<+\infty$, we define
$$K_p(f,t)=t\sup\Big\{ \left(\int_E f_*^p(\s)d\s\right)^{\frac1p}:|E|_\nu=t^{-p}\Big\}\qquad 
t\in ]0,1].$$ Then
$$K(f,t;L^{p,\infty},L^p)\approx K_p(f,t)$$
and
$$K_p(f,t)=t\left[\int_0^{t^{-p}}\psi_{*,\nu}(x)^pdx\right]^{\frac1p}$$
where $\psi(s)=s^{\frac1s}f_*(s)$, $\psi_{*,\nu}$ its decreasing rearrangement with respect to the measure $\nu$.
}
\end{theorem}
\pr
Let $f=g+h\in L^{p,\infty}+L^p$. Then, $f_*(s)\LEQ g_*\left(\dfrac s2\right)+h_*\left(\dfrac s2\right)$, for $s\in]0,1]$
\begin{equation}\label{eq3116}
K_p(f,t)\LEQ t\sup_{|E|_\nu=t^{-p}}\left(\int_E g_*\left(\dfrac s2\right)^p\right)^{\frac1p}+t\sup_{E:=|E|_\nu=t^{-p}}\left(\int_E h_*\left(\dfrac s2\right)^pds\right)^{\frac1p}.
\end{equation}
The first term can be bound as follows :
\def\dfsd{{\dfrac s2}}
\begin{equation}\label{eq3117}
t\sup_{|E|_\nu=t^{-p}}
\left(\int_E\left[s^\up g_*\left(\dfsd\right)\right]^p\dfrac{ds}s\right)
\lesssim t||g||_{L^{p,\infty}}\sup_{|E|_\nu=t^{-p}}|E|_\nu^\up=||g||_{L^{p,\infty}}
\end{equation}
While the second term satisfies
\begin{equation}\label{eq3118}
t\sup_{|E|_\nu=t^{-p}}\left(\int_Eh_*\left(\dfsd\right)^pds\right)^\up\LEQ t||h||_{L^p}.
\end{equation}
From the three last relations, we have
\begin{equation}\label{eq3119}
K_p(f,t)\lesssim||g||_{L^{p,\infty}}+t||h||_{L^p}.
\end{equation}
From which we derive
\begin{equation}\label{eq3120}
K_p(f,t)\lesssim K(f,t;L^{p,\infty},L^p).
\end{equation}
For the converse , let $t$ be fixed
and set $\psi(s)=s^\up f_*(s),\ \ s\in[0,1],\ \ \psi_{*,\nu}$ will denote its decreasing rearrangement  with respect to $\nu$, $A_t=\{s:\psi(s)>\psi_{*,\nu}(t^{-p})\}.$\\
By equimesurability, we have
$$|A_t|_\nu=t^{-p}.$$

Let us consider the measure preserving mapping $\s:\R\to (0,+\infty)$ such that $f=f_*\circ\s$  and set  $f_i=g_i\circ\s$, $i=1,2$ where, for $s\in(0,1)$
\begin{eqnarray*}
g_1(s)&=&s^{-\up}\psi_{*,\nu}(t^{-p})\chi_{A_t}(s)+f_*(s)\chi_{A_t^c}(s)\\
g_2(s)&=&s^{-\up}\Big(\psi(s)-\psi_{*,\nu}(t^{-p})\Big)\chi_{A_t}
\end{eqnarray*}
and $A^c_t$ is the complement of $A_t$ in $(0,1)$, say $A^c_t=\Big\{s:\psi(s)\LEQ\psi_{*,\nu}(t^{-p})\Big\}$.\\

Since $\s$ is measure preserving we have
$$||f_2||^p_{L^p}=||g_2||^p_{L^p}=\int_0^{|A_t|_\nu}
\Big(\psi_{*,\nu}(x)-\psi_{*,\nu}(t^{-p})\Big)^pdx.$$
From which we derive
\begin{equation}\label{eq3121}
||f_2||^p_{L^p}\LEQ\int_0^{t^{-p}}\psi_{*,\nu}(x)^pdx.
\end{equation}
While for $f_1$, we have
\begin{eqnarray}
||f_1||_{L^{p,\infty}}&=&||g_1||_{L^{p,\infty}}\LEQ\sup_s\left[\psi_{*,\nu}(t^{-p})\chi_{A_t}(s)
+s^\up f_*(s)\chi_{A_t^c}(s)\right]\nonumber\\
&\LEQ&\sup_s\left[\psi_{*,\nu}(t^{-p})\chi_{A_t}(s)+\psi(s)\chi_{A^c_t}(s)\right]\nonumber\\
&\LEQ&\psi_{*,\nu}(t^{-p})\hbox{ (by definition of }A_t^c).
\label{eq3122}
\end{eqnarray}
Since $f=f_1+f_2\in L^{p,\infty}+L^p$, we derive from relation (\ref{eq3121}) and (\ref{eq3122}) that
\begin{equation}\label{eq3123}
K(f,t;L^{p,\infty},L^p)\LEQ||f_1||_{L^{p,\infty}}+t||f_2||_{L^p}
\LEQ\psi_{*,\nu}(t^{-p})+t\left[\int_0^{t^{-p}}\psi_{*,\nu}^p(x)dx\right]^\up.
\end{equation}
Since the function $x\to \psi_{*,\nu}(x)$ is decreasing one has
\begin{equation}\label{eq3124}
\psi_{*,\nu}(t^{-p})\LEQ t\left[\int_0^{t^{-p}}\psi_{*,\nu}^p(x)dx\right]^\up.
\end{equation}
Thus, we derive from (\ref{eq3123}) and (\ref{eq3124})
\begin{equation}\label{eq3125}
K(f,t;L^{p,\infty},L^p)\lesssim t\left[\int_0^{t^{-p}}\psi^p_{*,\nu}(x)dx\right]^\up.
\end{equation}
Making use of the Hardy Littlewood (see\cite{RakoBook}), we have
\begin{equation}\label{eq3126}
\def\Max{\mathop{\rm Max\,}}
\int_0^{t^{-p}}\psi_{*,\nu}^p(x)dx=\Max_{|E|_\nu=t^{-p}}\int_E\psi^p(s)\dfrac {ds}s=\Max_{|E|_\nu=t^{-p}}\int_Ef_*(s)^pds.
\end{equation}
Thus 
\begin{equation}\label{eq3127}
K_p(f,t)=t\left[\int_0^{t^{-p}}\psi_{*,\nu}(x)^pdx\right]^{\frac1p}.
\end{equation}
This equality with relation (\ref{eq3125}) leads to
$$K(f,t;L^{p,\infty},L^p)\lesssim K_p(f,t).$$
\ \ \HF

As we noticed at the beginning, we recover the Maligranda-Persson's results stating that
$$(L^{p,\infty},L^p)_{\t,\frac p\t}=L^{p,\frac p\t}.$$
{\bf Proof of Maligranda-Persson's result}\\
One has, from the above result
$$||f||^{\frac p\t}_{(L^{p,\infty},L^p)_{\t,\frac p\t}}=\int_0^{+\infty}
\left[t^{1-\t}\left(\int_0^{t^{-p}}\Phi(x)dx\right)^{\frac1p}\right]^{\frac p\t}\dfrac{dt}t\dot=I_0$$
where we set temporarily $\Phi(x)=\psi_{*,\nu}(x)^p$.\\

Making the following change of variable $\s=t^{-p}$, we derive that 
$$I_0\approx\int_0^{+\infty}\left[\s^{\t-1}\left(\int_0^\s\Phi(x)dx\right)\right]^{\frac1\t}
\frac {d\s}\s=\int_0^{+\infty}\left[\dfrac1\s\int_0^\s\Phi(x)dx\right]^{\frac1\t}d\s,$$
but by the Hardy inequality this integral is equivalent to $\DST\int_0^{+\infty}\Phi(x)^{\frac1\t}dx$.\\
Therefore, we have 
$$I_0^{\frac \t p}\approx\left[\int_0^{+\infty}\psi_{*,\nu}(x)^{\frac p\t}dx\right]^{\frac\t p}
=\left[\int_0^1\Big[s^{\frac1p}f_*(s)\Big]^{\frac p\t}\dfrac{ds} s\right]^{\frac\t p}\quad\hbox{(by equimesurability)}.$$
This last quantity is equivalent to the norm of $f$ in $L^{p,\frac p\t}$.\HF

We end this section by proving Theorem 1.6 so we start with the following lemma~:
\begin{lemma}\label{L1LB}\ \\
Let $1<p<\infty.$ Then for any  $f\in L^{p,\infty}$ and all $0<t<1$,
$$  \sup\limits_{0<s<t} s^{\frac{1}{p}} f_\ast(s) \lesssim K(\rho(t),f;L^{p,\infty}, L^{(p} ),$$
where  $\rho(t)=(1-\ln t)^{-1+\frac{1}{p}}.$
\end{lemma}
\begin{proof} Fix $f\in L^{p,\infty}$ and $0<t<1.$ Set
$$ I(t,f)=\sup\limits_{0<s<t} s^{\frac{1}{p}} f_\ast(s).$$
It is sufficient to show that the following estimate
\begin{equation*}\label{ckt1e1}
I(t,f)\lesssim \|f_0\|_{L^{p,\infty}}+ \rho(t)\|f_1\|_{L^{(p}},
\end{equation*}
holds for  an arbitrary decomposition $f=f_0+f_1$ with $f_0\in L^{p,\infty}$ and $f_1\in L^{(p}.$ In view of the elementary inequality $f_\ast(t)\LEQ f_{0_\ast}(t/2)+f_{1_\ast}(t/2)$, it follows that $$I(t,f)\lesssim I(t/2,f_0)+I(t/2,f_1).$$
Clearly, $I(t/2,f_0)\LEQ \|f_0\|_{L^{p,\infty}}. $  Therefore, it remains to show that
\begin{equation}\label{ckt1e3}
I(t/2,f_1)\lesssim \rho(t)\|f_1\|_{L^{(p}}.
\end{equation}

Note that $\rho(t)\approx 1,\; 1/2<t<1.$ Therefore, (\ref{ckt1e3}) holds for all $1/2<t<1$ in view of the fact that $L^{(p}\hookrightarrow L^{p,\infty}.$ Next let $0<t<1/2$ and take $0<v<t/2.$  Then
\begin{eqnarray*}
\|f_1\|_{L^{(p}} &\GEQ& \int_{v}^{1}(1-\ln s)^{-1/p}\left(\int_{0}^{x}f_{1_\ast}(u)^pdu\right)^{\frac{1}{p}}\frac{ds}{s}\\
&\GEQ& \left(\int_{0}^{v}f_{1_\ast}(u)^p du\right)^{\frac{1}{p}} \int_{v}^{1}(1-\ln s)^{-1/p}\frac{ds}{s}\\
&\GEQ& v^{1/p}f_{1_\ast}(v)\int_{t}^{1}(1-\ln s)^{-1/p}\frac{ds}{s}\\
&\approx& (1-\ln t)^{1-1/p}v^{1/p}f_{1_\ast}(v),
\end{eqnarray*}
whence we obtain  (\ref{ckt1e3}) since $v$ was arbitrarily taken to be between $0$ and $t/2.$  The proof is complete.
\end{proof}
\begin{theorem}\label{t1}
Let $1<p<\infty$, $0 < \theta<1$, and  $1\LEQ r<\infty.$ Then, for any $f \in (L^{p,\infty}, L^{(p})_{\theta,r}, $ one has

$$ ||f||_{G\Gamma(\infty,r;v_1,v_2)} \lesssim ||f||_{(L^{p,\infty}, L^{(p})_{\theta,r}},$$
where $v_1(t)=t^{-1}(1-\ln t)^{r\theta(1-1/p)-1}$ and $v_2(t)= t^{1/p}.$
\end{theorem}
\begin{proof} Put $\rho(t)=(1-\ln t)^{-1+\frac{1}{p}}, \; 0<t<1.$ It immediately follows from Lemma \ref{L1LB} that

$$ \left(\int_{0}^{1}\left[\rho(t)\right]^{-\theta r} \left[ \sup_{0<s<t} s^{1/p}f_\ast(s)\right]^r \frac{ \rho^\prime (t) }{\rho(t)}dt\right)^{1/r}  \lesssim ||f||_{(L^{p,\infty}, L^{(p})_{\theta,r}}.$$
The simple observation
$$ \frac{ \rho^\prime (t) }{\rho(t)}\approx t^{-1}(1-\ln t)^{-1} $$
completes the proof.
\end{proof}

Now the estimate resulting from Theorem 1.3 in \cite{FFGKR} is:
\begin{theorem}\label{t2}\ \\
Let $1< p<\infty$,  $0 < \theta<1$, and $1\LEQ r<\infty.$ Then
$$\|f\|_{G\Gamma(p,r;w_1,w_2)}\lesssim \|f\|_{(L^{p,\infty}, L^{(p} )_{\theta,r}},$$
where  $w_1(t)=t^{-1}(1-\ln t)^{r\theta-1}$ and $w_2(t)= (1-\ln t)^{-1}.$
\end{theorem}
The G$\Gamma$ spaces in Theorems \ref{t1} say $G\Gamma(\infty,r;v_1,v_2) $ and \ref{t2} say $G\Gamma(p,r,w_1,w_2)$  are not comparable. Thus we get Theorem 1.6.

\section{\bf \large  Some interpolation inequalities for Small and Grand Lebesgue spaces}
One may combine the above results with some standard results on interpolation spaces to deduce few inequalities as Property \ref{propri1}.\\
We recall the following result that can be found in \cite{Tartar}.

\begin{theorem}\label{t700}\ \\
{\it
Let $E_0$ and $E_1$ be two Banach spaces continuously embedded into some topological 
vector space. \\
For $0\LEQ\theta\LEQ1$, one has
$$(E_0,E_1)_{\theta,1}\subset E\subset E_0+E_1$$
if and only if
$$\hbox{ there exists $c>0:\forall\,a\in E_0\cap E_1\qquad||a||_E\LEQ c||a||^{1-\theta}_0||a||_1^\theta$}$$
where $||\cdot||_i$ denotes the norm in $E_i$, $i=0,1$.
}
\end{theorem}
{\bf Proof of Property \ref{propri1}}\\
We apply the above Theorem \ref{t700} with $E_0=L^{n')}$,  $E_1=L^{(n'}$.\\
Then from Theorem $\ref{tt2}$ and Corollary \ref{c2} of Proposition \ref{p21}  one has
$$\big(L^{n')},L^{(n'}\big)_{\alpha,1}=L^{(n',\beta}$$ with $\beta=n\alpha-n+1$ and $\alpha>\dfrac1{n'}$.\\
Since $L^{(n',\beta}\subset L^{n',\infty}\subset L^{n')}$, we deduce the result from Theorem $\ref{t700} $ with $E=L^{n',\infty},\ \theta=\alpha$.\\

The same argument holds for the second inequality, since
$$\big(L^{n'},L^{(n'}\big)_{\alpha,1}=L^{(n',\alpha}\subset E=\Big(L^{n',\infty},L^{(n'}\Big)_{\alpha,1}.$$
\ \ \HF
\section{\bf A remark on the associate space of $G\Gamma(p,r;\omega)$}
\begin{theorem}\ \\
 Let $1<p<\infty$, $1<r<\infty$ and $\delta>0.$ Put $w(t)=t^{-1}(1-\ln t)^{\delta-1}.$  Then
$$ \left[G\Gamma(p,r;w)\right]^\prime=G\Gamma(p^\prime,r^\prime;w_1,w_2),$$
where $w_1(t)=t^{-1}(1-\ln t)^{\frac{r^\prime \delta}{r}-1}$ and $w_2(t)=(1-\ln t)^{-\frac{2 p^\prime \delta}{r}}$.
\end{theorem}
\begin{proof}\ \\
Put $\eta=p^\prime \delta/r,$ and take $\alpha=2\eta,$ $\beta=\eta/2$ and $\theta=1/3.$ Then by \cite{AHG} Theorem $7$, we have
$$ G\Gamma(p,r;w)=(L^{(p,\beta},L^{(p,\alpha})_{\theta,r}.$$
Thus, using duality relation of real interpolation spaces (see, for instance, \cite{Bergh-Lofstrom} Theorem $3.7.1$), we get
\begin{eqnarray*}
 \left[G\Gamma(p,r;w)\right]^\prime&=&(L^{p^\prime),\beta},L^{p^\prime),\alpha})_{\theta,r^\prime}\\
 &=&(L^{p^\prime),\alpha},L^{p^\prime),\beta})_{1-\theta,r^\prime},
\end{eqnarray*}
finally, an application of Theorem 3.6  completes the proof.
\end{proof}
\begin{remark}{\em In view of Lemma $2.4$ , we get the following equivalent norm on $\left[G\Gamma(p,r;w)\right]^\prime$:
$$\|f\|_{\left[G\Gamma(p,r;w)\right]^\prime}\approx \left(\int_{0}^{1}(1-\ln t)^{ -\frac{r^\prime \delta}{r}-1}\left(\int_{t}^{1}f_\ast(x)^pdx\right)^{\frac{r^\prime}{p^\prime}}\frac{dt}{t}\right)^{1/r^\prime}, $$
which is apparently simpler than the one which follows from \cite{GPS} Theorem $1.1$ (vi).
}
\end{remark}
{\bf Acknowledgment }  : \\
The third author, M.R. Formica has been partially supported by Gruppo Nazionale per l'Analisi Matematica, la Probabilit\`a e le loro Applicazioni (GNAMPA) of the Istituto Nazionale di Alta Matematica (INdAM) and by Universit\`a degli Studi di Napoli Parthenope through the project "sostegno alla Ricerca individuale".\\
The work of fourth author, A.Gogatishvili , has been partially supported by Shota Rustaveli National Science Foundation of Georgia (SRNSFG) [grand number FR17-589], by Grant 18-00580S of the Czech Science Foundation and RVO:67985840.
We address a grateful thanks to Professor Teresa Signes to point out an error in the first version of this work.

\ \\

\end{document}